\documentclass[reqno,12pt,a4paper,oneside]{amsart}

\usepackage{amsmath}
\usepackage{array}
\usepackage{mathrsfs}
\usepackage{amssymb}
\usepackage{amsfonts}
\usepackage{latexsym}
\usepackage{amsthm}
\usepackage{amsmath}
\usepackage[arrow, matrix, curve]{xy}
\usepackage{graphicx}	
\usepackage{tikz}
\usetikzlibrary{matrix,arrows}
\input xy

\theoremstyle{plain}
\newtheorem{theorem}{\bf Theorem}[section]
\newtheorem{lemma}[theorem]{\bf Lemma}

\newtheorem{proposition}[theorem]{\bf Proposition}

\def\a{\alpha}         
          
  \def\cC{{\mathcal C}}       
         
\def\d{\delta}

\def\k{\kappa}         
\def\l{\lambda} \def\cM{{\mathcal M}}       
 \def\cN{{\mathcal N}}       
\def\m{\mu}            
\def\n{\nu}            
\def\r{\rho}           

\def\o{\omega}

\def\Z{{\mathbb Z}}    
 \def\N{{\mathbb N}}  
 \def\Q{{\mathbb Q}}

\def\no{\noindent}

\newcommand{\Mg}{\overline{\mathcal{M}}_g}
\newcommand{\Mgnn}{\overline{\mathcal{M}}_{g,2n}}
\newcommand{\Mgn}{\overline{\mathcal{M}}_{g,n}}
\newcommand{\Mgm}{\overline{\mathcal{M}}_{g,m}}
\newcommand{\Mgun}{\overline{\mathcal{M}}_{g+n}}

\newcommand{\Ngn}{\overline{\mathcal{N}}_{g,n}}
\newcommand{\M}{\overline{\mathcal{M}}}

\newcommand{\tNgn}{\tilde{\mathcal{N}}_{g,n}}
\newcommand{\dirr}{\d_{\mathrm{irr}}} 

\def\Pic{\mathop{\mathrm{Pic }}}
\def\Sym{\mbox{Sym}}

\let\geq\geqslant
\let\leq\leqslant
\def\eq{\begin{equation}}
\def\qe{\end{equation}}

\newcommand\blfootnote[1]{%
  \begingroup
  \renewcommand\thefootnote{}\footnote{#1}%
  \addtocounter{footnote}{-1}%
  \endgroup
}

\begin{document}

\title{On the Kodaira dimension of the moduli space of nodal curves}

\author{Irene Schwarz}
\address{Humboldt Universit\"at Berlin, Institut f\"ur Mathematik, 
Rudower Chausee 25, 12489 Berlin, Germany}
\email{schwarzi@math.hu-berlin.de}
\begin{abstract}
\no We show that the compactification of the moduli space of {\em $n-$nodal curves of geometric genus g}, i.e. $\Ngn:= \Mgnn /G$, with $G:=(\Z_2)^n\rtimes S_n$,  is of general type for $g \geq 24$, for all $n \in \N$. While this is a fairly easy result,  it requires completely different techniques to extend it to low genus $5 \leq g \leq 23$. Here we need
that the number of nodes  varies in a band
 $n_{\mathrm{min}}(g) \leq n \leq  n_{\mathrm{max}}(g)$, where $n_{\mathrm{max}}(g)$ is the largest integer smaller than (or in some cases equal to) $\frac{7}{2}(g-1)-3$.
The lower bound  $n_{\mathrm{min}}(g) $ is close to the bound found in \cite{l1}, \cite{f} for $\Mgnn$ to be of general type (in many cases it is identical). This  will be tabled in Theorem  \ref{two} which is the main result of this paper.

\end{abstract}

\thanks{This paper is part of my PhD Thesis written under the supervision of Prof. G. Farkas at Humboldt Universit\"at Berlin, Institut f\"ur Mathematik.  I wish to thank my advisor for suggesting this topic and for his efficient support. I also thank my family for their help in putting this document into acceptable English, and I acknowledge helpful comments of Will Sevin on an earlier version of this paper which led to its present form.}


\maketitle

\setcounter{page}{1}

\section{Introduction}

The goal of this paper is to study the moduli space $\cN_{g,n}:= \cM_{g,2n}/G$ of {\em $n-$nodal curves of geometric genus g} and its compactification $\Ngn:= \Mgnn /G$. 
Here $\cM_{g,2n}$ is the moduli space of smooth curves of genus $g$ with $2n$ distinct marked points   
on which the group $G:=(\Z_2)^n\rtimes S_n$, with  $S_n$ being the symmetric group, acts in the following way: We group the labels $\{1,...,2n\}$ in pairs 
$(i,n+i),$ where  $ i\leq n$, corresponding to pairs of points $(x_i,y_i)$ (where $y_i=x_{i+n}$). Each of the $n$ copies of $\Z_2$ acts 
on one of these pairs by switching components, and $S_n$ acts by permutation of the pairs. 
Then $G$ acts on the moduli space $\Mgnn$ by acting on the labels of the $2n$ marked points.
We shall henceforth assume without further comment that this grouping is fixed whenever we consider $2n$ marked points. Clearly, $G$ then acts transitively  on the marked points   and  fixed point free on the moduli space.

A special interest in nodal curves, and thus in the study of the moduli space
$\Ngn$, (with the aim of better understanding the general properties of smooth curves) is a common feature of all deformation type arguments. They go back at least to 
Severi who proposed  in \cite{se} to use  $g-$nodal curves 
to prove the Brill-Noether theorem, see  \cite{bn}. Such deformation type methods have become prominent in modern algebraic geometry, e.g. in the systematic theory of limit linear series (see \cite{eh2}) which progressively simplified the earlier deformation type arguments in the first rigorous proof of the Brill-Noether theorem and the Gieseker-Petri theorem (see \cite{gh}, \cite{g},  \cite{eh0}, \cite{eh1},  

\cite{sch1}).

Thus it seems natural  
to study the moduli space $\cN_{g,n}$ of nodal curves in its own right.\\

A crucial point in our analysis is the following diagram (which, incidentally, also shows that $\Ngn$ actually parametrizes $n-$nodal curves with geometric genus $g$ and arithmetic genus $g+n$):
 
  $$\begin{tikzpicture} [baseline=(current  bounding  box.center)]
\matrix (m) [matrix of math nodes,row sep=2em,column sep=3em,minimum width=2em]
{\Mgnn& 	\Mgun\\
 	    \Ngn	&	\\};
\path[->>]
   (m-1-1) edge node[left] {$\pi$} (m-2-1);
\path[->]
	(m-1-1) edge node [above] {$\chi$} (m-1-2)
 	(m-2-1) edge node [below]{$\m$} (m-1-2);
         
\end{tikzpicture}$$
which factors $\chi$ through the quotient map $\pi$, $\chi$ being the map which glues $x_i$ to $y_i$ for every pair $(x_i,y_i)$. Thus each pair is transformed into a nodal point of the curve $C$ under consideration, increasing its arithmetic genus by $n$.

While it was known for a long time that the moduli space $\Mg$ of algebraic curves is uniruled for low genus $g$, Eisenbud, Harris and Mumford showed in the 1980's that $\Mg$ is of general type for $g \geq 24$ (see  \cite{hm} and \cite{eh}), the case of $\overline{\mathcal{M}}_{23}$ being somewhat special (see also \cite{f2}).

Since then there has been a lot of research refining this picture of the birational geometry of moduli spaces of curves. The addition of marked points on the algebraic curve leads to the moduli space $\Mgn$ of $n$-pointed curves of genus $g$ whose Kodaira dimension was studied in \cite{l1}, with some improvements in \cite{f}, 
employing refined versions of the original techniques. Heuristically, additional marked points push the moduli space in the direction of being of general type. In particular,  $\overline{\mathcal{M}}_{23,n}$ is of general type for all  $n\geq 1$, and even for $4\leq g\leq 22$  the moduli space $\Mgn$ achieves general type if $n\geq n_{\mathrm{min}}(g)$ for $n_{\mathrm{min}}(g)$ sufficiently large.
In another direction, the paper \cite{bfv} studies the Kodaira dimension of the Picard variety $\overline{\mbox{Pic}}^d_g$ parametrizing line bundles of degree $d$ on curves of genus $g$.\\

A further gratifying aspect of $\Ngn=\Mgnn/G$ is its structure as a quotient of the well-understood space $\Mgnn$ by a finite group. Such quotients have been studied before in various contexts, and their birational geometry might be different from what one could naively expect.

For instance, the space $\overline{\cC}_{g,n}:=\Mgn /S_n$, for $g$ large,is  of general type for $n \leq g-1$ but uniruled for $n \geq g+1$; only the transitional case $g=n$ is challenging (see \cite{fv1}). In fact, to see this for large $n$, one observes that the fibre of $\Mgn/S_{n}\to \Mg$ over a smooth curve $[C]\in \Mg$ is birational to the symmetric product $C_{n}$. Since the Riemann-Roch theorem implies that any effective divisor of degree $d> g$ lies in some $\mathfrak{g}^1_d$, the quotient  $\Mgn/S_{n}$ is trivially uniruled for $n> g$. 

This proof, of course, uses the very special structure of the above fibre as a symmetric product $C_n$ (which can be interpreted as a family of divisors on $C$ and thus is accessible via Brill-Noether theory) and not just abstract properties of the group $S_n$. Still, it points to the possibility that taking the quotient with respect to a sufficiently large group might somehow destroy the property of an algebraic variety of being of general type. At least for low genus, this is compatible with the findings of this paper, see Theorem \ref{two} below. In our case, such a phenomenon might be related to the existence of an upper finite bound $n_{\mathrm{max}}(g)$ for the number of nodes allowed on the geometric genus $g$ curve 
Note that our group {$G\subset S_{2n}$ is a subgroup of $S_{2n}$, implying that $\Mgnn\to \Mgnn/S_{2n}$ factors through $\Ngn$. Thus the space $\Ngn$ should be somewhat intermediary between
$\Mgnn$ and  $\overline{\cC}_{g,2n}$.

Making this precise  is the central result of the present paper. Most demanding is the result for small $g$, namely:

\begin{theorem}   \label{two}
$\Ngn$ is of general type in the special cases 
$5 \leq g \leq 23,$  and\\
$n_{\mathrm{min}}(g) \leq n \leq  n_{\mathrm{max}}(g)$ 
with  $n_{\mathrm{min}}(g), n_{\mathrm{max}}(g)$  given in the following table:
$$\begin{array}{c|c|c|c|c|c|c|  c|c|c|c|c|  c|c|c|c|c|  c|c|c}
g &5&6&7&8&9&10&11&12&13&14&15&16&17&18&19&20&21&22&23\\
\hline
n_{\mathrm{min}}&9&9&8&8&8&6&6&6&6&5&6&5&5&5&4&4&2&2&1\\
n_{\mathrm{max}}&10&14&18&21&25&28&32&35&38&42&46&49&52&56&60&63&66&70&74
\end{array}$$
\end{theorem}

For large $g$ we will show the following:

\begin{theorem} \label{one}
The moduli space $\Ngn$ is of general type for $g\geq 24$.
\end{theorem}

The crucial reason for Theorem \ref{one} is that $\Mg$ is of general type for $g\geq 24$. The theorem could automatically be improved if there were proof of $\Mg$ being of general type for some $g<24$. Then $\Ngn$ also is of general type, by an identical proof.
We emphasize that in this case our result is uniform in  the number $n$ of nodes. There is no finite upper bound $n_{\mathrm{max}}(g)$. The proof (presented in Section \ref{p1}) uses arguments of fairly general nature and does not require elaborate calculations with carefully selected divisors.
This changes drastically in low genus $g \leq23$.
It is clear from the above diagram that $\Ngn$ can only be of general type if $\Mgnn$ is. 
In our opinion it is easier to show that $\Mgnn$ is of general type than to establish this property for $\Ngn$. Thus we have not even attempted to consider those values of $g,n$ where $\Mgnn$ is not known to be of general type.
Thus, for low genus, we have considered individually the cases in which $\Mgnn$ is known to be of general type (see \cite{l1} and \cite{f}).
For the convenience of the reader we shall recall:\\

$\Mgnn$ is of general type 
for $5 \leq g \leq 23,$  and
$n \geq n_{\mathrm{min}}(g) , $ 
with  $n_{\mathrm{min}}(g$) 
being given in the following table:
$$\begin{array}{c|c|c|c|c|c|c|  c|c|c|c|c|  c|c|c|c|c|  c|c|c}
g &5&6&7&8&9&10&11&12&13&14&15&16&17&18&19&20&21&22&23\\
\hline
n_{\mathrm{min}}&8&8&8&7&7&6&6&6&6&5&5&5&5&5&4&3&2&2&1

\end{array}$$
\\

In order to obtain the analogue  table for $\Ngn$ in Theorem \ref{two},
our approach  involves deriving sufficient conditions for the property of being of general type in terms of divisors invariant under the group action. Basic building blocks are the Brill-Noether divisors on $\Mg$ and $\Mgun$, appropriately pulled back, and divisors of Weierstrass-type (see Section 4). For $g+1$ or $g+n+1$ prime, the Brill-Noether divisors have to be replaced by less efficient divisors which we take from \cite{eh}. This gives, in any case,  a system of linear inequalities which is solvable for certain values of  $(g,n)$. This solvability is a sufficient condition for $\Ngn$ being of general type. In spirit our analysis is related to \cite{fv1,fv2,fv3}, but  the detailed analysis is  different and requires new ideas.  
The explicit solution of the relevant systems of inequalities is best done using computer algebra, paying special attention to a number of limiting cases where the appropriate divisors have to be chosen with care.

We further remark that, in the  table of Theorem \ref{two}, the upper cut-off at $n_{\mathrm{max}}(g)$ is the largest integer smaller than (or in some cases equal to) $\frac{7}{2}(g-1)-3$.
and that, using this divisor based computational approach to prove general type in high genus $g \geq 24$ - avoiding the completely different arguments in the proof of Theorem \ref{one} - would give a much weaker result than Theorem \ref{one},  since one only gets general type for $n$ bounded by the same $n_{\mathrm{max}}(g)$. This appearance of an upper bound for $n$ in Theorem \ref{two} is the main difference to both our Theorem \ref{one} and the results in the low-genus table for $\Mgn$.   We remark that in order to improve our result on the upper cut-off with the techniques of this paper, one needs a new type of effective divisor $A$ on $\Mgnn$  of the form (using our notation established in Section 4)
$$ a \lambda - b_{irr} \d_{irr} - \sum_{i,S} b_{i,S} \d_{i,S},$$
with all coefficients nonnegative and $b_{0,S} \neq 0$ for $|S|=2$, or of the form
$$-a \lambda +c \psi - b_{irr} \d_{irr} - \sum_{i,S} b_{i,S} \d_{i,S},$$
with all coefficients nonnegative and $b_{0,S} > 3 c$ for $|S|=2$. We do not know of any such divisor. In any case, at present it is not clear if our result is sharp.
To clarify this question is, in our opinion, the most interesting point left open by the results of our present paper.

 The outline of our paper is as follows. In Section 2 we prove Theorem \ref{one}. In Section \ref{sc} we give sufficient conditions for our statement on $\Ngn$ being of general type in terms of an appropriate decomposition of the canonical divisor $K_{\Ngn}$.  
 In Section \ref{div} we introduce most of the divisors needed in the subsequent analysis (i.e. all the standard divisors used in the next section), while in Section \ref{system} we derive most of the table (the standard part of it) of Theorem \ref{two} by  solving  the ensuing system of inequalities. This proof is the least technical. 
It leaves open a number of special cases (precisely 9), each of which requires additional special divisors and a lot of special attention. This is the content of Section \ref{special} and concludes the proof of Theorem \ref{two}.

Finally we remark that, though not being strictly necessary, it  is 
 convenient to employ in Sections 5-6 a standard computer algebra program, e.g. Mathematica or Maple.

\section{Proof of Theorem \ref{one}} \label{p1}
\setcounter{equation}{0}
We shall prove the assertion by representing $\Ngn$ via a base and a general fibre. Recall that for any surjective proper morphism of normal projective varieties $\phi: Y \to X$ with general fibre $F$ the Kodaira dimension satisfies $\k(Y) \geq \k(X)+\k(F)$. In particular, if $F$ and $X$ are of general type, then $Y$ is of general type.

Now consider the projection 
$$\phi: \Ngn= \Mgnn/G \to\Mg,$$
with general fibre 
$$F  \simeq C^{2n}/G \simeq \Sym ^{n}(\Sym^2(C)) =(C^2/\Z_2)^n/S_n,$$ 
where $ \Sym^m(C):= C^m/S_m$ denotes 
the $m$-fold symmetric product of the curve, and the crepant resolution given by  the Hilbert scheme (see e.g. \cite{hl})
$$\mbox{Hilb}^n(S) \to \Sym^n(\Sym^2(C)), \qquad S=\Sym^2(C).$$
This is of general type if $S$ is of general type. It is well known that $\Sym^d(C)$ is of general type for $d\leq g-1$. Combining these results we get that both the base $\Mg$ and the general fibre $F$ are  of general type, completing the proof.
For the sake of the reader we recall the argument that $\Sym^d(C)$ is of general type for $d\leq g-1$ based on results of \cite{kou}. 

Letting $C$ be an irreducible smooth genus $g$ curve, we consider for $d\leq g-1$ the canonical projection
$$\pi : C^d \to C_d:=\Sym^d(C),$$
which induces a canonical map $D\mapsto \mathcal{L}_D$ from divisors on $C$ to divisors on $C_d$. Denoting by $K$ the canonical divisor on $C$ and by $\Delta /2$ the ramification divisor of $\pi$, Prop. 2.6 of \cite{kou} gives the canonical divisor on $C_d$ as
\begin{equation}\label{canonical class Cd}
K_{C_d}=\mathcal{L}_K - \Delta /2.
\end{equation}
Using the standard interpretation of points in $C_d$ as effective divisors on $C$ of degree $d$ and fixing some effective divisor $D$ of degree $g-1-d$ gives a map
$$\alpha_D: C_d \to C_{g-1}, \quad A \mapsto A+D.$$
Composing $\a_D$ with the Abel-Jacobi map
$$u: C_{g-1}\to J^{g-1}(C),$$
and taking $\theta$ to be the $\theta$-divisor on the Jacobi torus $J^{g-1}(C)$, Prop. 2.3 and Prop. 2.7. of \cite{kou} give
$$u^* \a_D^* \theta =\mathcal{L}_K-\Delta /2 -\mathcal{L}_D.$$
Combining this with \eqref{canonical class Cd} yields
$$ K_{C_d}= u^* \a_D^* \theta +\mathcal{L}_D,$$
representing the canonical divisor of $C_n$ as a sum of an ample divisor (because $\theta$ is ample) and an effective divisor. This is a well known criterion for $C_d$ to be of general type (see e.g. our Section \ref{sc} below).

\section{Moduli spaces of nodal curves} \label{sc}
\setcounter{equation}{0}

The aim of this section is to develop a sufficient condition for $\Ngn$ being of general type. This requires a basic understanding of the Picard group $\Pic( \Ngn)$ and an explicit description of the  boundary divisors and tautological classes on $\Ngn$ which we shall always consider as $G$-invariant divisors on $\Mgnn$ (any such divisor descends to a divisor on $\Ngn$).
For results on $\Mgn$ we refer to the book \cite{acg} (containing in particular the relevant results from the papers \cite{ac1} and \cite{ac2}). All Picard groups are taken with rational coefficients and, in particular, we identify the Picard group on the moduli stack with that of the corresponding coarse moduli space.
\\

In particular, we recall the notion of the Hodge class $\l$ on $\Mgn$, which automatically is $G-$ invariant and thus  gives the Hodge class $\l$ on $\Ngn$ (where, by the usual abuse of notation, we denote both classes by the same symbol).

In order to describe the relevant boundary divisors on $\Mgn$, we recall that $\Delta_0$ (sometimes also called $\Delta_{\mathrm{irr}}$) on 
$\Mg$ is the boundary component consisting of all (classes of) stable curves of arithmetical genus $g$, having at least one nodal point with the property that ungluing the curve at this node preserves connectedness. Furthermore, $\Delta_i$,  for $1 \leq i \leq \lfloor\frac{g}{2}\rfloor,$ denotes the boundary component of curves possessing a node of order $i$ (i.e. ungluing at this point decomposes the curve in two connected components of arithmetical genus $i$ and $g-i$ respectively). Similarly, on $\Mgn$ and for any subset $S \subset \{1, \ldots,n\}$, we denote by $\Delta_{i,S}, 0 \leq i \leq \lfloor\frac{g}{2}\rfloor,$ the boundary component consisting of curves possessing a node of order $i$ such that after ungluing the connected component of genus $i$ contains precisely the marked points labeled by $S$. Note that, if $S$ contains at most 1 point, one has $\Delta_{0,S}= \emptyset$ (the existence of infinitely many automorphisms on the projective line technically violates stability). Thus, in that case, we shall henceforth consider $\Delta_{0,S}$ as the zero divisor.

We shall denote by $\delta_i, \delta_{i,S}$ the rational divisor classes of $\Delta_i, \Delta_{i,S}$ in $\Pic \Mg$ and $\Pic \Mgn$, respectively. Note that $\delta_0$ is also called $\delta_{\mathrm{irr}}$ in the literature, but we shall reserve the notation $\delta_{\mathrm{irr}}$ for the pull-back of $\d_0$ under the forgetful map
$\phi:\Mgn \to \Mg$.

Next we recall the notion of the point bundles $\psi_i, 1 \leq i \leq n,$ on $\Mgn$. Informally, the line bundle $\psi_i$ (sometimes called the cotangent class corresponding to the label $i$) is given by choosing as fibre of $\psi_i$ over a point $[C;x_1, \ldots, x_n]$ of $ \Mgn$ the cotangent line $T_{x_i}^v(C)$. For later use we also set
\begin{equation}  \label{basechangeprep}
\omega_i:= \psi_i - \sum_{S \subset \{1, \ldots,n \}, S \ni i} \d_{0,S},
\end{equation}
and introduce  $\psi=\sum_{i=1}^n \psi_i.$  Clearly, on $\Mgnn,$ the class $\psi$ is $G-$ invariant.\\

We shall now consider the action of $G$ on these divisor classes on $\Mgnn.$ Clearly, $\l, \psi$ and $\delta_{\mathrm{irr}}$ are $G -$ invariant. Furthermore, $\d_{i,S}$ is mapped into $\d_{j,T}$ by some element of $G$ if and only if $i=j, |S|=|T|$ and $S$ and $T$ contain the same number of pairs. This implies that one gets a $G-$ invariant  divisor class by setting
$$\d_{i;a,b}= \sum_S \d_{i,S},$$
where the sum is taken over all subsets $S$ containing precisely $a$ pairs and $b$ single points.\\

As a first step in the direction of our sufficient criterion we need the following result on the geometry of the moduli space $\Ngn.$

\begin{theorem}  \label{noadcon}
The moduli space $\Ngn$ has only {\em canonical singularities} or in other words:
The singularities of $\Ngn$ do not impose {\em adjunction conditions}, i.e. if 
$\r: \tNgn  \to \Ngn$ is a resolution of singularities, then for any  $\ell \in \N$
there is an isomorphism
\begin{equation}
\r^*: H^0((\Ngn)_{\mathrm{reg}}, K_{(\Ngn)_{\mathrm{reg}}}^{\otimes \ell}) \to H^0(\tNgn,K_{\tNgn}^{\otimes \ell}).
\end{equation}
Here $(\Ngn)_{\mathrm{reg}}$ denotes the set of regular points of $\Ngn$, considered as a projective variety, and $K_{\tNgn},  K_{(\Ngn)_{\mathrm{reg}}}$ denote the canonical classes on $\tNgn$ and  $(\Ngn)_{\mathrm{reg}}$.
\end{theorem}

The proof  
follows the lines of the 
the proof of Theorem 1.1 in \cite{fv1}. 
We shall briefly review the argument. A crucial input is Theorem 2 of the seminal paper \cite{hm} which proves that the moduli space $\Mg$ has only canonical singularities. The proof relies on the Reid-Tai criterion: Pluricanonical forms (i.e. sections of $K^{\otimes \ell}$) extend to the resolution of singularities, if for any automorphism $\sigma$ of an object of the moduli space the so-called {\em age} satisfies  $age(\sigma) \geq 1$. The proof in \cite{fv1} then proceeds to verify the Reid-Tai criterion for the quotient of $\Mgn$ by the full symmetric group $S_n $. Here one specifically has to consider those automorphisms of a given curve which act as a permutation of the marked points. For all those automorphisms the proof in \cite{fv1} verifies the   Reid-Tai criterion. Thus, in particular, the criterion is verified for all automorphisms which act on the marked points as an element of some subgroup of $S_n$. Thus, the proof in \cite{fv1} actually establishes the existence of only canonical singularities for {\em any} quotient $\Mgn/G$ where $G$ is a subgroup of $S_n$. Clearly, this covers our case.

Theorem \ref{noadcon} implies that the Kodaira dimension of $\Ngn$ equals the Kodaira-Iitaka dimension of the canonical class $K_{\Ngn}$. In particular, $\Ngn$ is of general type if  $K_{\Ngn}$ is a positive linear combination of an ample and an effective  rational class on $\Ngn$.  It is convenient to slightly reformulate this result. We need

\begin{proposition} \label{psi}
The class $\psi$ on $\Mgnn$ is the pull-back of a divisor class on $\Ngn$ which is big and nef.
\end{proposition}
 
\begin{proof}
 Farkas and Verra have proven in  Proposition 1.2 of \cite{fv2} that the $S_{2n}$-invariant class $\psi$ descends to a big and nef divisor class $N_{g,2n}$ on the quotient space $\Mgnn /S_{2n}$. Consider the sequence of  natural projections
 $\Mgnn \xrightarrow{\pi} \Mgnn/G \xrightarrow{\n} \Mgnn/S_{2n}.$
Then $\n^*(N_{g,2n})$ is a big and nef divisor on $\Ngn=\Mgnn/G$ and $\pi^*(\n^*(N_{g,2n})=\psi$. 
\end{proof}

 Now observe that the ramification divisor (class) of the quotient map $\pi$ is precisely $\d_{0;1,0}$. In fact a quotient map $X\to X/G$ is ramified in a point $x\in X$ exactly if there exists a non-trivial element $g\in G$ such that $g(x)=x$. In our case an element $g\in G$ acts on a class of pointed curves by $g[C,x_1,\ldots, x_{2n}]=[C;x_{g(1)},\ldots, x_{g(2n)}]$. By standard results  the pointed curves $(C,x_1,\ldots, x_n)$ and $(C;x_{\sigma (1)},\ldots, x_{\sigma (2n)})$ are isomorphic for some $\sigma\neq id \in S_n$ , if and only if $C$ has a rational component with exactly two marked points and $\sigma$ is the transposition switching these two points. Since our group $G$ contains exactly those transpositions belonging to the pairs $(x_i,y_i)$, the ramification divisor is $\d_{0;1,0}$.
 
 Furthermore, standard results on the pullback give

\begin{equation}\label{K}
K:= \pi^*(K_{\Ngn})= K_{\Mgnn} - \mbox{ram } \pi= 13 \l + \psi - 2 \d -\d_{1,\emptyset}- \d_{0;1,0}.
\end{equation} 

Here $\d$ is the class of the boundary of $\Mgnn$, i.e.
$$ \d = \dirr + \sum_{0 \leq i \leq \lfloor\frac{g}{2}\rfloor} \sum_{S \subset \{ 1, \ldots ,2n \} }
\d_{i,S}. $$

 We thus obtain the final form of our sufficient condition:
 If  $K$ is a positive multiple of $\psi$ + some effective $G - $ invariant divisor class on $\Mgnn$, then $\Ngn$ is of general type.

\section{Effective $G$-invariant divisors on $\Mgnn$} \label{div}
\setcounter{equation}{0}

In this section we construct $G -$ invariant effective divisors on $\Mgnn.$ First we recall the following standard result.

\begin{proposition}  \label{divisor}
Let $f:X \to Y$ be a morphism of projective schemes, $D \subset Y$ be an effective divisor and assume that $f(X)$ is not contained in $D$. Then $f^*(D)$ is an effective divisor on $X$.
\end{proposition}
\

The most natural effective divisors on $\Mg$ are the Brill-Noether divisors parametrizing all curves $C$ possessing a $\mathfrak{g}^r_d$ for fixed $g,r,d$ with Brill-Noether number $\r(g,r,d)=-1$. These exist as long as $g+1$ is composite.
We recall from \cite{eh}:

\begin{lemma}
Assume that $g+1$ is not prime and fix some integers $r,s\geq 1$ such that $g+1=(r+1)(s-1).$ Then 
$$\mathfrak{BN}_g:=\{[C]\in \mathcal{M}_g | C \mbox{ carries a } \mathfrak{g} ^r_{rs-1} \}$$
is an (effective) divisor on $\mathcal{M}_g$. Furthermore the class of its compactification as a $\Q$-divisor is given by
$$
[\overline{\mathfrak{ BN}}_g]= c \Big( (g+3)\l -\frac{g+1}{6} \d_0 -\sum_{i= 1}^{\lfloor\frac{g}{2}\rfloor} i(g-i)\d_i \Big),
$$
for some positive rational number $c$.
\end{lemma}

Note that only the constant $c$ depends on the choice of $r$ and $s$. Note also that  $g+1=(r+1)(s-1)$ implies that the Brill-Noether number satisfies $\r(g,r,rs-1)=-1$. Therefore $\mathfrak{BN}_g$ parametrises precisely those (non-general) curves violating the condition of the Brill-Noether Theorem (see \cite{gh}), which states that a general curve carries a $\mathfrak{g}^r_d$, if and only if $\r(g,r,d)\geq 0$.

Now consider the forgetful map $\phi:\Mgnn \to \Mg$ and recall from the introduction  the map $\chi : \Mgnn \to \Mgun$.

If neither $g+1$ nor $g+n+1$ are prime, we pull back the class of $[\overline{\mathfrak{ BN}}_g]$ as a $\Q$-Divisor (up to multiplication with a positive scalar) and set
\begin{equation} \label{pullback1}
B=\frac{1}{c}\phi^*[\overline{\mathfrak{ BN}}_g],  \qquad D=\frac{1}{c} \chi^*[\overline{\mathfrak{ BN}}_{g+n}].
\end{equation}

We shall now apply results of \cite{ac2}, where the pull-backs of the divisor classes $\psi_i,\dirr,\d_{i,S}$ are computed under the forgetful map forgetting one point and under the gluing map gluing two marked points. Thus repeatedly applying  Lemma 1.2 and Lemma 1.3 of \cite{ac2} yields

\begin{lemma}
With the above notation, one has
\begin{equation}\label{B}
B=b_{\l}\l +0\cdot\psi-  b_0\dirr -\sum_{i= 1}^{\lfloor\frac{g}{2}\rfloor}\sum_S b_i\d_{i,S}
\end{equation}
with $$ b_\l=g+3,\qquad b_0=\frac{g+1}{6},\qquad b_i=i(g-i).$$ 
Furthermore
\begin{equation}\label{D}
D=d_\l\l+d_0(\psi-\dirr-\sum_{i=0}^{\lfloor\frac{g}{2}\rfloor}\sum_{a+b\leq n,b\neq0}\d_{i;a,b})-\sum_{i=0}^{\lfloor\frac{g}{2}\rfloor}\sum_{a=0}^nd_{i+a}\d_{i;a,0}
\end{equation}
with $$ d_\l=g+n+3,\qquad d_0=\frac{g+n+1}{6},\qquad d_i=i(g+n-i).$$

\end{lemma}

In particular, both divisors are $G -$invariant.
The problem is that these divisors will not exist for $g+1$ or $g+n+1$ prime. In that case we have to use less efficient divisors parametrizing curves that violate the Gieseker-Petri condition.

We recall that a curve C satisfies the Gieseker-Petri condition, if for every line bundle $L$ the natural map 
$$\mu: H^0(C,L)\otimes H^0(C,K\otimes L^{-1}) \to H^0(C,K)$$ is injective.

We also recall from \cite{eh} Theorem 2 the following divisor.

\begin{lemma}
Assume that $g=2d-2$ is even, then
$$\mathfrak{GP}_g:=\{ [C]\in\mathcal{M}_g| C \mbox{ violates the Gieseker-Petri condition}\}$$ is an (effective) divisor on $\mathcal{M}_g$.
Furthermore the class of its compactification is given by 
$$
[\overline{\mathfrak{GP}}_g]= c( a_{\l}\l- \sum_{i=0}^{\frac{g}{2}} a_i \d_i),
$$ for some rational number $c>0$ and
$$ a_{\l}=6d^2+d-6, \quad a_0= d(d-1), \quad a_1=(2d-3)(3d-2), \quad a_{i+1}>a_i. $$
\end{lemma}

Similar to \eqref{pullback1},  we pull back $[\overline{\mathfrak{GP}}]$ along the forgetful map $\phi$ and the gluing map $\chi$  and set
\begin{equation}
E=\frac{1}{c}\phi^*[\overline{\mathfrak{GP}}_g],\qquad F=\frac{1}{c}\chi^*[\overline{\mathfrak{GP}}_{g+n}].
\end{equation}  
By repeatedly applying \cite{ac2}, Lemma 1.2 and Lemma 1.3, we obtain
\begin{equation}\label{E}
E=e_{\l}\l+0\cdot\psi-e_0\dirr -\sum_{i=1}^{\frac{g}{2}}\sum_S e_i\d_{i,S},
\end{equation}
with 
$$e_{\l}=6(\frac{g}{2}+1)^2+\frac{g}{2}-5, \quad
e_0=(\frac{g}{2}+1)\frac{g}{2}, \quad
e_1=(g-1)(\frac{3g}{2}+1), \quad e_{i+1}>b_i,$$ 
and
\begin{equation}\label{F}
F=f_{\l}\l+f_0(\psi-\dirr-\sum_{i=0}^{\lfloor\frac{g}{2}\rfloor}\sum_{a+b\leq n,b\neq0}\d_{i;a,b})-\sum_{i=0}^{\lfloor\frac{g}{2}\rfloor}\sum_{a=0}^n f_{i+a}\d_{i;a,0},
\end{equation}
with 
$$f_{\l}=6(\frac{g+n}{2}+1)^2+\frac{g+n}{2}-5, \quad
f_0=(\frac{g+n}{2}+1)\frac{g+n}{2}, \quad
f_1=(g+n-1)(\frac{3g+3n}{2}+1), $$
$$f_{i+1}>\tilde{b}_i.$$

Now observe that the divisor classes $B$ and $E$ are effective by Proposition \ref{divisor}, because the forgetful map $\phi$ is onto. The divisor class $F$ is effective because the general nodal curve is Gieseker-Petri general. This is already implicitly contained in \cite{g} (see also \cite{eh1}), in the proof of the Gieseker-Petri Theorem. Since the proof uses deformation to a nodal curve, it actually shows the above statement. 
Effectiveness of $D$ follows in the same way observing that the Gieseker-Petri condition implies the Brill-Noether condition. Alternatively, one could consider the proof of the Brill-Noether Theorem in \cite{gh}, which also uses deformation to a nodal curve.

Finally we need divisors of Weierstrass-type. We recall from \cite{l1}, Section 5,
the divisors $W(g;a_1, \ldots,a_m)$ on $\Mgm$, where $a_i \geq 1$ and $\sum a_i=g$.
They are given by the locus of curves $C$ with marked points $p_1, \ldots, p_m$ such that there exists a $\mathfrak{g}^1_g$ on $C$ containing $\sum_{1 \leq i \leq m}a_i p_i$. We want to minimize the distance between the weights $a_i$. Thus we decompose
$g=km+r$, with  $r<m$,  and set
\begin{equation}\label{wmdef}
W_m=W(g;a_1, \ldots,a_m), \qquad a_j=k+1 \, \, (1\leq j \leq r), \quad  
a_j=k  \, \, (r+1 \leq j \leq m).
\end{equation}

This gives, in view of \cite{l1}, Theorem 5.4,
\begin{equation}\label{wm}
\begin{split}
W_m 
& = -\l+	\sum_{i=1}^r \frac{(k+1)(k+2)}{2}\omega_i +\sum_{i=r+1}^m \frac{k(k+1)}{2} \omega_i  -0\cdot\dirr 
\\&-\sum_{i,j\leq r} (k+1)^2 \d_{0,\{i,j\} }  -\sum_{i\leq r,j>r} k(k+1) \d_{0,\{i,j\} } -\sum_{i,j>r} k^2 \d_{0,\{i,j\} } \\
&-\mbox{higher order boundary terms},  
\end{split}
\end{equation}
where {\em higher order}  means a positive linear combination of $\d_{i,S}$ where either $i>0$ or $|S|>2$.

From $W_m$ we want to generate  a $G - $ invariant divisor class $W$ on $\Mgnn$, by summing over appropriate pullbacks. Thus we let $S,T$ be disjoint subsets of $\{1,\ldots,2n \}$ with $|S|=r$ and $|T|=m-r$ (recall that $r$ is fixed by the decomposition   $g=mk+r$) and let 
\begin{equation}
\phi_{S,T}: \Mgnn \to \Mgm
\end{equation}
be a projection (i.e. a surjective morphism of projective varieties) mapping the class 
$[C; q_1, \ldots ,q_{2n}]$ to $[C; p_1, \ldots ,p_{m}],$
where the points $q_i$ labeled by $S$ are sent to the points $p_1, \ldots,p_r$ (all with weights $a_i=k+1$) and the points labeled by $T$ are sent to the points $p_{r+1}, \ldots,p_m$ (all with weights equal to $k$). Clearly, for fixed $g$, there are precisely  $$a(g,n,m,r):=\binom{2n}{r} \binom{2n-r}{m-r}$$ such projections. With this notation we introduce
\begin{equation}\label{W}
\begin{split}
W := &\sum_{S,T} \phi_{S,T}^* W_m \\
= &-w_\l \l +w_\psi \psi+0\cdot\dirr-\sum_{s\geq 2} w_s \sum_{|S|=s} \d_{0,S}
\\&-\mbox{higher order boundary terms},  
\end{split}
\end{equation}
where  {\em higher order} denotes a positive linear combination of boundary divisors $\d_{i,S}$ with     $i \geq 1$,
\begin{equation}  \label{ws}
w_s \geq s w_\psi \geq 3w_\psi \qquad \mbox{for } s\geq 3,
\end{equation}
\begin{equation}  \label{w lambda}
w_\l= a(g,n,m,r) =-\binom{2n}{r} \binom{2n-r}{m-r},
\end{equation}
\begin{equation}\label{w psi}
 w_\psi=\binom{2n-1}{r-1} \binom{2n-r}{m-r} \dfrac{(k+1)(k+2)}{2} +\binom{2n-1}{r} \binom{2n-r-1}{m-r-1} \dfrac{k(k+1)}{2},
\end{equation}
\begin{equation}\label{w2}
\begin{split}
w_2= &2w_\psi + \binom{2n-2}{r-2} \binom{2n-r}{m-r}(k+1)^2 +2\binom{2n-2}{r-1} \binom{2n-r-1}{m-r-1}k(k+1) 
 \\&+\binom{2n-2}{r} \binom{2n-r-2}{m-r-2}k^2.  
\end{split}
\end{equation}

 Equation \eqref{w psi} is proved by applying pullback to \eqref{wm}, using $\o:=\sum_{i=1}^{2n} \o_i$,
 \begin{equation}
\sum_{S,T} \sum_{i=1}^r  \pi_{S,T}^* \o_i = a(g,n,m,r) \frac{r}{2n} \o =\binom{2n-1}{r-1} \binom{2n-r}{m-r}\o
 \end{equation}
and
\begin{equation}
\sum_{S,T} \sum_{i=r+1}^m  \pi_{S,T}^* \o_i = a(n,m,g) \frac{m-r}{2n} \o=\binom{2n-1}{r} \binom{2n-r-1}{m-r-1},
\end{equation}
noting that equation \eqref{basechangeprep} implies
\begin{equation} \label{basechange}
\o=\psi - \sum_S |S| \d_{0,S}.
\end{equation}

The sums over the pullbacks of the boundary divisors are computed by similar combinatorial considerations which we leave to the reader. Note that both the summand
$2w_\psi$ on the right hand side of \eqref{w2}  and the bound in \eqref{ws} are  generated by the change of basis given in \eqref{basechangeprep}.\\

It turns out that we will be able to improve the upper bound on $n$ given in Proposition
\ref{prop one} by replacing $W$ with divisors parametrizing curves, that fail the so called {\em Minimal Resolution conjecture}, see \cite{f}, Theorem 4.2. 

\begin{lemma}
Fix integers $g,r\geq 1$ and $0\leq k\leq g$ and set $m=(2r+1)(g-1)-2k$. Then
$$\mathfrak{Mrc}_{g,k}^r:= \{ [C,x_1,\ldots, x_m]\in \mathcal{M}_{g,m}| h^1(C, \wedge^i M_{K_C}\otimes K_C^{\otimes(r+1)} \otimes\mathcal{O}_C(-x_1-\cdots-x_m))\geq 1\}$$
is a divisor on $\mathcal{M}_{g,m}$. Furthermore the class of its compactification is given by
\begin{equation}
[\overline{\mathfrak{Mrc}}_{g,k}^r]= \frac{1}{g-1} \binom{g-1}{k}(-a_\l \l +a_\psi \psi+a_{\mathrm{irr}}\dirr -\sum_{i,s} a_{i,s} \sum_{|S|=s} \d_{i,S} )
\end{equation}
where 
\begin{equation}\label{coefficients MinResConj}
\begin{split}
&a_\l	=\frac{1}{g-2}\Big( (g-1)(g-2)(6r^2+6r+r)+k(24r+10k+10-10g-12rg)\Big),	\\ 
&a_\psi 	= rg+g-k-r-1,	\\
&a_{\mathrm{irr}}=\frac{1}{g-2}\Big( \binom{r+1}{2} (g-1)(g-2)+k(k+1+2r-rg-g)\big), \\
&a_{0,s}	= \binom{s+1}{2}(g-1)+s(rg-r-k) \qquad\mbox{and} \qquad a_{i,s}\geq a_{0,s}.
\end{split}
\end{equation}
\end{lemma}

Note that $[\overline{\mathfrak{Mrc}}_{g,k}^r]$ is already $G$-invariant. \\

For $g$ odd and $2n\geq g-1$ we can always find an $r,k$ such that $2n=m=(2r+1)(g-1)-2k$. In fact we will only  consider $2n\geq g+1$ and choose $k\leq g-2$. This will make some computations easier and uniquely determine $r,k$ for given $g,n$.
So for $g$ odd and such $r,k$ we define the divisor $U$ as the class of $\overline{\mathfrak{Mrc}}_{g,k}^r $ as a $\Q$-divisor (up to multiplication with a positive scalar)by

\begin{equation} \label{U}
U:=[\overline{\mathfrak{Mrc}}_{g,k}^r ]= -u_\l+u_\psi \psi +u_{\mathrm{irr}} \dirr -\sum_{i,s} u_{i,s} \sum_{|S|=s} \d_{i,S}
\end{equation}
with $u_\l, u_\psi, u_{\mathrm{irr}}, u_{i,s}$ being $a_\l, a_\psi, a_{\mathrm{irr}}, a_{i,s}$ from \eqref{coefficients MinResConj}.

For $g$ even we set $2n-1=m=(2r+1)(g-1)-2k$ and pull back via all possible forgetful maps $\phi_i:\Mgnn\to\Mgm$, forgetting the $i$th point. Again we only consider $k\leq g-2$ corresponding to $m\geq g+1$. We emphasize that, just as above  for $U$, the integers $r$ and $k$ are uniquely determined by $g$ and $n$. This gives
\begin{equation} \label{V}
V:=\sum_{i=1}^{2n}\phi^*_i [\overline{\mathfrak{Mrc}}_{g,k}^r]= -v_\l+v_\psi \psi +v_{\mathrm{irr}} \dirr -\sum_{i,s} v_{i,s} \sum_{|S|=s} \d_{i,S}
\end{equation}
with 

\begin{equation}
\begin{split}
v_\l	=&\frac{2n}{g-2}\Big( (g-1)(g-2)(6r^2+6r+r)+k(24r+10k+10-10g-12rg)\Big), \\
v_\psi	=& (2n-1)(rg+g-k-r-1), \\
v_{\mathrm{irr}}=& \frac{2n}{g-2}\Big( \binom{r+1}{2} (g-1)(g-2)+k(k+1+2r-rg-g)\big), \\
v_{0,2} =& 2 (rg + g - k - r - 1) + (2 n - 2) (3 g - 3 + 2 rg - 2 r - 2 k) \\
v_{i,s}\geq & v_{0,2}.
\end{split}
\end{equation}

\section{Standard cases in the proof of Theorem \ref{two}: Reduction to a system of inequalities} \label{system}
\setcounter{equation}{0}

As mentioned in Section 3, the moduli space $\Ngn$ is of general type, if and only if its canonical divisor $K_{\Ngn}$ is the sum of an ample divisor and effective divisors. We show this by decomposing $K=\pi^*(K_{\Ngn})$ on $\Mgnn$ as a sum of a positive multiple of $\psi$, of non-negative multiples of the $G$-invariant divisors constructed in the last section and of non-negative multiples of $\l,\dirr$ and the $\d_{i;a,b}$.

We will begin by using the divisor class $W$ to show:
\begin{proposition}   \label{prop one}
$\Ngn$ is of general type in the special cases 
$7 \leq g \leq 23,$  and\\
$n_{\mathrm{min}}(g) \leq n \leq  n_{\mathrm{max}}(g)$ 
with  $n_{\mathrm{min}}(g), n_{\mathrm{max}}(g)$  given in the following table:
$$\begin{array}{c|c|c|c|c|  c|c|c|c|c|  c|c|c|c|c|  c|c|c}
g &7&8&9&10&11&12&13&14&15&16&17&18&19&20&21&22&23\\
\hline
n_{\mathrm{min}}&9&8&8&8&6&7&6&6&6&6&5&6&4&4&3&4&1\\
n_{\mathrm{max}}&10&12&14&16&18&20&22&24&26&28&30&32&34&36&38&40&42
\end{array}$$
\end{proposition} 

Then we will use the divisor classes $U$ and $V$ to improve the upper bound $n_{\mathrm{max}}$ and extend our result to $g=5,6$.

\begin{proposition}   \label{prop two}
$\Ngn$ is of general type in the special cases 
$5 \leq g \leq 23,$  and\\
$n_{\mathrm{min}}(g) \leq n \leq  n_{\mathrm{max}}(g)$ 
with  $n_{\mathrm{min}}(g), n_{\mathrm{max}}(g)$  given in the following table:
$$\begin{array}{c|c|c|c|c|c|c|  c|c|c|c|c|  c|c|c|c|c|  c|c|c}
g &5&6&7&8&9&10&11&12&13&14&15&16&17&18&19&20&21&22&23\\
\hline
n_{\mathrm{min}}&9&9&8&8&8&8&6&7&6&6&6&6&5&6&4&4&3&4&1\\
n_{\mathrm{max}}&10&14&18&21&25&28&32&35&38&42&46&49&52&56&60&63&66&70&74
\end{array}$$
\end{proposition}

Note that the lower bound $n_{\mathrm{min}}$ differs from Theorem \ref{two} only in the cases $g=5,6,7$. For $g=5$ and $g=6$ this is due to the fact that the preliminary upper bound actually lies below the lower bound.\\

\begin{proof} [Proof of Proposition \ref{prop one}] For $g+1$ and $g+n+1$ both composite, we will search for a non-negative linear combination of the divisors $B,D$ and $W$ such that 
$$K-xB-yD-zW$$ 
(possibly up to an arbitrarily small multiple $\epsilon \psi$, see equations
\eqref{1}, \eqref{2} below  and  our discussion of Case II)
becomes an effective combination of the tautological and boundary classes. (Whenever $g+1$ is prime we replace $B$ by $E$ and whenever $g+n+1$ is prime we replace $D$ by $F$.) This will translate into a system of inequalities, one for each tautological or boundary class. However it is easy to check that only the inequalities imposed by $\l,\psi,\dirr,\d_{0;10}$ and $\d_{0;02}$ are relevant. The others are automatically satisfied as soon as these 5 are. 

The table of relevant coefficients is:

\begin{equation}  \label{table}
\begin{array}{c|c|c|c|c|c}
  		& \l			&\psi			&\dirr		&\d_{0;1,0}	 & \d_{0;0,2}		\\
  		\hline 
 B	& b_\l	&0		& -b_0	& 0		& 0 	\\
 D	& d_\l	&d_0 	& -d_0  &-d_1 	& -d_0  \\
 E  & e_\l  &0 		& -e_0 	& 0 	&0 		\\
 F 	& f_\l 	&f_0 	& -f_0 	&-f_1 	& -f_0 	\\
 W	& -w_\l &w_\psi & 0 	&-w_2	&-w_2	\\
 K	&13		&1		&-2		&-3		&-2		\\
\end{array}
\end{equation}
with 
$$b_\l=g+3, \qquad b_0=\frac{g+1}{6},$$

$$d_\l=g+3, \qquad d_0=\frac{g+1}{6},\qquad d_1=g+n-1, $$

$$e_\l=6(\frac{g}{2}+1)^2+\frac{g}{2}-5, \quad
e_0=(\frac{g}{2}+1)\frac{g}{2},$$ 

$$f_\l=6(\frac{g+n}{2}+1)^2+\frac{g+n}{2}-5, \quad
f_0=(\frac{g+n}{2}+1)\frac{g+n}{2}, \quad
f_1=(g+n-1)(\frac{3g+3n}{2}+1), $$
and $w_\l, w_\psi, w_2$ from \eqref{w lambda}, \eqref{w psi} and \eqref{w2}.

When looking at the divisor $W$ defined in equation \eqref{W} we have the condition $m\leq\min (2n,g)$. For the remainder of the paper we shall always set $m=\min(2n,g)$ and $g=mk+r$ with $r<m$.
We shall treat separately the (easy)  

Case I: $2n < g$ (which we shall split into $2n \leq g-2$ and $2n=g-1$) 

and the (difficult) 

Case II: $g\leq 2n$. 

Looking at the coefficients $w_\psi$ and $w_2$  of this divisor  (see \eqref{w psi} and\eqref{w2}) it is easy to show that $w_2>3w_\psi$ for $2n\leq g-2$, $w_2=3w_\psi$ for $2n=g-1$ or $2n=g$ and $w_2<3w_\psi$ for $2n\geq g$. This motivates treating $2n=g-1$ separately.\\

For $2n\leq g-2$ and $g+1$ composite, we have $3w_\psi<w_2$ and therefore the optimal decomposition is $x=\frac{2}{b_0}, \quad y=0, z=\frac{3}{w_2}$. This is exactly the same decomposition Logan uses to prove $\Mgn$ being of general type. It satisfies the inequalities imposed by $\dirr$ and $\d_{0;10}$ as equalities and also satisfies the inequalities imposed by $\psi$ and $\d_{0;0,2}$. Only the inequality imposed by $\l$ remains to be checked:
\begin{equation} \label{lambda}
(\l): \qquad \frac{2b_\l}{b_0}-\frac{3w_\l}{w_2}\leq 13
\end{equation}
Whenever this is true for a given $g$ and $n$, we get the decomposition 
\begin{equation}\label{decomposition 1}
 K=\frac{2}{b_0}B+\frac{3}{w_2}W +(1-\frac{3w_\psi}{w_2})\psi +\mbox{"effective"}.
\end{equation}

For $2n\leq g-2$ and $g+1$ prime (and thus $g\geq 28$), we replace $B$ by $E$ in \eqref{lambda}.  Thus, whenever we have 
\begin{equation} \label{cond}
(\l): \qquad \frac{2e_\l}{e_0}-\frac{3w_\l}{w_2}\leq 13,
\end{equation}
 we actually get the desired decomposition 
\begin{equation}\label{decomposition 1 prime}
 K=\frac{2}{e_0}E+\frac{3}{w_2}W +(1-\frac{3w_\psi}{w_2})\psi +\mbox{"effective"},
\end{equation}  
proving that $\Ngn$ is of general type. Checking the inequality \eqref{cond} by use of a small computer program establishes the table for $2n \leq g-2$.\\

For $2n=g-1$ the above decomposition no longer works, because then $3w_\psi=w_2$, 
giving a vanishing coefficient of $\psi$. We therefore need to use the divisor $D$ for $g+n+1$ composite and $F$ for $g+n+1$ prime. Note that $g+1=2n+2$ is composite.
Whenever 
\begin{equation} \label{lambda1}
(\l): \qquad \frac{2b_\l}{b_0}-\frac{3w_\l}{w_2}< 13
\end{equation}
we can use the decomposition
\begin{equation} \label{1}
K=\frac{2}{b_0}B+\frac{2\epsilon}{d_0}D+\frac{1-3\epsilon}{w_\psi}W +\epsilon\psi +\mbox{"effective"}
\end{equation}
or
\begin{equation} \label{2}
K=\frac{2}{b_0}B+\frac{3\epsilon}{f_0}F+\frac{1-4\epsilon}{w_\psi}W +\epsilon\psi +\mbox{"effective"}
\end{equation}
for some $\epsilon >0 $ sufficiently small.

This proves Proposition
\ref{prop one} in Case I.\\

We shall now treat the  difficult Case II:   $2n\geq g$. Let us begin by assuming both $g+1$ and $g+n+1$ are composite. We start similarly to Case I, by trying to decompose the canonical class as 
$$K=xB+yD+zW+\epsilon\psi+\mbox{"effective"}$$
for some non-negative $x,y,z$ and positive $\epsilon$.

The coefficients of $W$ become especially easy:
$$w_\l=\binom{2n}{g}, \qquad w_\psi=\binom{2n-1}{g-1}, \qquad w_2=2\binom{2n-1}{g-1} +\binom{2n.-2}{g-2}.$$

We shall first show that there is a finite upper bound for $n$, 
i.e.  for fixed $g$ and $n$ sufficiently large, the conditions imposed by $\psi,\d_{0;1,0}$ and $\d_{0;0,2}$ cannot be satisfied simultaneously. 
In fact, let us consider the corresponding system of 3 linear inequalities for the variables $y,z$, read off the table \eqref{table}:
\begin{equation} \label{systemineq}
\begin{split}
(\psi):      	\hspace{1,1cm}	 d_0 y + w_\psi z       &< 1   \\
(\d_{0;1,0}): 	\qquad 		-d_1 y -w_2 z  			&\leq -3  \\
(\d_{0;0,2}): 	\qquad		-d_0 y -w_2 z			&\leq -2
\end{split}
\end{equation}

We apply   Gaussian elimination, allowing only positive multiples of the 
 inequalities in \eqref{systemineq}. Note that (for $2n \geq g$) any solution $y,z$ of 
\eqref{systemineq}  automatically is positive: $(\psi) + (\d_{0;0,2})$ implies $z>0$ and $3 (\psi) + (\d_{0;1,0]})$ then gives $y>0.$
 Then the inequality $(\psi)+ (\d_{0;0,2})$ gives 
\begin{equation}\label{lower bound z}
z > \frac{1}{w_2-w_\psi}\geq \frac{1}{2w_\psi}
\end{equation}
and inequality $d_1(\psi)+d_0(\d_{0;1,0})$ gives 
\begin{equation}
z < \frac{d_1-3d_0}{d_1 w_\psi- d_0 w_2}.
\end{equation}

This is solvable if and only if 
\begin{equation}\label{condition system ineq}
(d_1-3d_0)(w_2-w_\psi)-(d_1 w_\psi -d_0 w_2)>0.
\end{equation}

Consider
\begin{equation}
\begin{split}
p_g(n)&:=(12n-6)\binom{2n-1}{g-1}^{-1} \Big((d_1-3d_0)(w_2-w_\psi)-(d_1 w_\psi -d_0 w_2)\Big)\\
	&=-2n^2+(2g-5)n+g^2-11g+9	
\end{split}
\end{equation}
as a polynomial in $n$. Then $p_g(n)>0$ if and only if
\begin{equation}
\frac{1}{4}(2g-5-\sqrt{97-108g+36g^2}) <n< \frac{1}{4}(2g-5+\sqrt{97-108g+36g^2}).
\end{equation}

As we are only interested in solutions with  $n \in \N,$ we get

\begin{equation}
n \leq \frac{1}{4}(2g-5+\sqrt{81-108g+36g^2})= 2g-4.
\end{equation}
This establishes the existence of an upper bound for $n$, as claimed above.

Now let us assume, that $\frac{g}{2}\leq n\leq 2g-4$ and that $(y,z)$ is a solution of the three inequalities \eqref{systemineq} (preferably with z as large as possible).
 We set 
$$x=\frac{1}{b_0}(2-d_0 y)$$ 
and show that $(x,y,z)$ satisfies also the two remaining inequalities:

\begin{equation} 
\begin{split}
(\l): \qquad  b_\l x+d_\l y-w_\l z &\leq 13 \\
(\dirr): \hspace{1,8cm} -b_0 x -d_0 y   &\leq -2
\end{split}
\end{equation}

In fact $(\dirr)$ will be satisfied as equality, and we are left with checking $(\l)$.

If $g+1$ is prime we replace $B$ by $E$ and for $g+n+1$ prime we replace $D$ by $E$. As there are only finitely many cases to check, we have done this 
by explicitly computing the decomposition with the help of a simple program in Mathematica.
 This proves all the results  tabled in Proposition \ref{prop one}.
\end{proof}

In order to show Proposition \ref{prop two} we replace the divisor $W$ by either $U$ (for $g$ odd) or by $V$ (for $g$ even). 

\begin{proof}[Proof of Proposition \ref{prop two}]
As  a general remark, note that the divisors $U,V$ do not exist if $n$ is too small, depending on $g$, i.e. for $2n < g-4$. In this case, we simply use Proposition \ref{prop one} to generate the table in Proposition \ref{prop two}.

In all other cases we want to decompose $K$ as a 
sum of $xB$ or $xE$, $yD$ or $yF$, $zU$ or $zV$ with non-negative $x,y,z$, some positive multiple $\epsilon \psi$ of the point bundles and an effective combination of the tautological classes and boundary divisors.
So let us consider the table of relevant coefficients
\begin{equation}  \label{table2}
\begin{array}{c|c|c|c|c|c}
  		& \l			&\psi			&\dirr		&\d_{0;1,0}	 & \d_{0;0,2}		\\
  		\hline 
 B	& b_\l	&0		& -b_0			& 0		& 0 	\\
 D	& d_\l	&d_0 	& -d_0  		&-d_1 	& -d_0  \\
 E  & e_\l  &0 		& -e_0 			& 0 	&0 		\\
 F 	& f_\l 	&f_0 	& -f_0 			&-f_1 	& -f_0 	\\
 U	& -u_\l &u_\psi & u_{\mathrm{irr}}	&-u_{0,2}&-u_{0,2}	\\
 V	& -v_\l &v_\psi & v_{\mathrm{irr}}	&-v_{0,2}&-v_{0,2}	\\
 K	&13		&1		&-2		&-3		&-2		\\
\end{array}
\end{equation}
where $b_\l,b_0$ are given in \eqref{B}, $d_\l,d_0,d_1$ in \eqref{D}, $e_\l,e_0$ in \eqref{E}, $f_\l,f_0,f_1$ in \eqref{F}, $u_\l$, $u_\psi$,$u_{\mathrm{irr}}$,$ u_{0,2}$ in \eqref{U} and $v_\l, v_\psi,v_{\mathrm{irr}}, v_{0,2}$ in \eqref{V}.

$\l,\psi,\dirr,\d_{0;1,0},d_{0;0,2}$ each determine a linear inequality that can be read off from the table. 
Analogous to \eqref{systemineq} we begin by considering the system of inequalities 

\begin{equation} \label{systemieq odd composite}
\begin{split}
(\psi):      	\hspace{1,35cm}	 d_0 y + u_\psi z       &< 1   \\
(\d_{0;1,0}): 	\qquad 		-d_1 y -u_{0,2} z  			&\leq -3  \\
(\d_{0;0,2}): 	\qquad		-d_0 y -u_{0,2} z			&\leq -2.
\end{split}
\end{equation}

This corresponds to the case $g+1$ composite, $g+n+1$ composite and $g$ odd. It is one of 6 possible cases and will be the only one which we shall discuss in detail.

By Gaussian elimination (compare \eqref{condition system ineq})  this is solvable, if and only if
\begin{equation}
\begin{split}
0 	&\leq p(g,n,r,k):=6\Big((d_1-3d_0)(u_{0,2}-u_\psi)-(d_1 u_\psi -d_0 u_{0,2})\Big)  \\
	&=9 - 12 g + 3 g^2 + k + g k - 3 n + 3 g n + k n + r - g^2 r + n r - 
 g n r
\end{split}
\end{equation}
This inequality defines the upper cut-off in the table of Proposition \ref{prop two} (in the special case considered here), which improves the table of Proposition \ref{prop one}. We the have to show that for all these values of $g$ and $n$ we actually get a solution of our system of inequalities. This we did by computing, in all these cases, an explicit solution using Mathematica.

All other cases are treated in a similar way: For $g$ even,  the divisor $U$ has to be replaced by $V$, if $g+n+1$ is prime, one has to replace $D$ by $F$, and if $g+1$ is prime, $B$ nedds to by replaced by $E$. In each of these cases  one gets a corresponding system of linear inequalities and a neccessary upper cut-off for $n$. This upper bound defines $n_{\mathrm{max}}(g)$ in our table. 
As above, we have then computed a solution of the ensuing system of inequalities by Mathematica, for all relevant values of  $g$ and $n$.
This gives the improved table in Proposition  \ref{prop two}.

\end{proof}

\section{Proof of Theorem \ref{two}: Special computations} \label{special}
\setcounter{equation}{0}

The aim of this section is to prove Theorem \ref{two}. In view of Propostion \ref{prop two} this boils down to specific calculations for each of the values of $(g,n)$ contained in the table of Theorem \ref{two}, but not in the table of Proposition \ref{prop two}. In each case, we shall need additional divisors, specifically adapted to the case at hand. Our standard references for these divisors are \cite{l1}, \cite{f}.\\

\textbf{$\mathcal{N}_{10,6}$ and $\mathcal{N}_{10,7}$} \\
We use the divisor class $\mathcal{Z}_{10,0}$ on $\M_{10}$ from \cite{f},  Theorem 1.4. This is a divisor that violates the slope conjecture, i.e. it has a smaller slope than the Brill-Noether divisor. We have $s(\mathcal{Z}_{10,0})=7$, i.e $\mathcal{Z}_{10,0}=c( 7\l-\dirr -\sum_i a_i \d_i)$ on $\M_{10}$.
Take $Z=\phi^*(\mathcal{Z}_{10,0})$, the pull-back via the forgetful map $\phi:\M_{10,2n}\to\M_{10}$.
Then for $n=6$ we can decompose $K$ as an effective combination of $Z,F,W$,  some positive multiple of $\psi$ and an effective combination of the tautological and boundary classes.
For $n=6$ we can decompose $K$ as an effective combination of $Z,D,W$, some  positive multiple of $\psi$ and an effective combination of the tautological and boundary classes.\\

\textbf{$\mathcal{N}_{21,2}$} \\
We use the divisor class $\mathcal{Z}_{21,0}$ on $\M_{21}$ from \cite{f} Theorem 1.4. This divisor has slope $s(\mathcal{Z}_{10,0})=\frac{2459}{377}$, which is smaller than the slope of the Brill-Noether divisor. We set $Z=\phi^*(\mathcal{Z}_{21,0})$ the pull-back via the forgetful map $\phi:\M_{10,4}\to\M_{10}$.
The canonical divisor will decompose as an 
effective combination of $Z,W$ some positive multiple of $\psi$ and an effective combination of the tautological and boundary classes.\\

\textbf{$\mathcal{N}_{16,5}$} \\
Similarly to the above divisors, we use the divisor class $\mathcal{Z}_{16,1}$ on $\M_{16}$ from \cite{f}, Corollary 1.3. This divisor has slope $s(\mathcal{Z}_{10,0})=\frac{407}{61}$, which once again is smaller than the slope of the Brill-Noether divisor. We set $Z=\phi^*(\mathcal{Z}_{16,1})$, the pull-back via the forgetful map $\phi:\M_{16,10}\to\M_{16}$.
The canonical divisor will decompose as an 
effective combination of $Z,W$, some positive multiple of $\psi$ and an effective combination of the tautological and boundary classes.\\

\textbf{$\mathcal{N}_{12,6}$} \\
Here we use the divisor class $\mathcal{D}_{12}=13245\l-1926\dirr-9867\d_1-\ldots $ on $\M_{12}$ from \cite{fv4} and pull it back along the forgetful map $\phi:\M_{12,12}\to\M_{12}$.
The canonical divisor will decompose as an 
effective combination of $\phi^*(\mathcal{D}),F,W$ some positive multiple of $\psi$ and an effective combination of the tautological and boundary classes.\\

\textbf{$\mathcal{N}_{22,2}$ and $\mathcal{N}_{22,3}$} \\
For $g=22$ there are no Brill-Noether divisors, because $22+1$ is prime. Instead of using the divisor class $E$ we take a Brill-Noether divisor $\mathcal{B}_{23}=c( 26\l-4\dirr-22\d_1-\ldots) $ on $\M_{23}$. We pull this divisor back along $\chi:\M_{22,2}\to \M_{23}$ and then along all possible forgetful maps $\phi_S: \M_{22,4}\to \M_{22,2}$ or $\M_{22,6}\to\M_{22,2}$.

On $\M_{22,4}$, adding all these pull-backs, gives us the divisor class $L_{22,4}=12c(13\l+\psi-2\dirr-\frac{10}{3}\sum_{|S|=2}\d_{0,S}-\ldots)$. This divisor class alone proves the canonical divisor $K$ to be effective. By using a linear combination with $\epsilon_1 E, \epsilon_2 W$ for some suitable $\epsilon_1,\epsilon_2>0$ we can show $K$ to be big.

On $\M_{22,6}$, adding all these pull-backs, gives us the divisor class $L_{22,6}=30c(13\l+\frac{2}{3}\psi-2\dirr-\frac{56}{30}\sum_{|S|=2}\d_{0,S}-\ldots)$.
Using $L_{22,6}$ and $W$ we can show $K$ to be big.\\

\textbf{$\mathcal{N}_{14,5}$} \\
We can show that $\mathcal{N}_{14,5}$ is of general type by using the divisor classes $B$ and $\mathfrak{Nfold}^1_{14,12} = c(35\l+54\psi-10\dirr-173\sum_{|S|=2}\d_{0,S}-\ldots)$ from \cite{f}, Theorem 4.9.\\

\textbf{$\mathcal{N}_{18,5}$} \\
We  take the divisor class $\mathfrak{Lin}^8_{24}=c(290\l+24\psi-45\dirr -82\sum_{|S|=2}\d_{0,2}-\ldots)$ on $\M_{18,9}$ from \cite{f}, Theorem 4.6.
We pull this divisor class back along all forgetful maps $\phi_i: \M_{18,10} \to \M_{18,9}$. Using the sum of these pullbacks $D$ and $W$, we can show that  $\mathcal{N}_{18,5}$ is of general type.\\

\end{document}